\documentclass[12pt]{article}
\usepackage{amsmath}
\usepackage{latexsym}
\usepackage{amssymb}
\newtheorem{thm}{Theorem}[section]
\newtheorem{la}[thm]{Lemma}
\newtheorem{Defn}[thm]{Definition}
\newtheorem{Problm}[thm]{Problem}
\newtheorem{Remark}[thm]{Remark}
\newtheorem{prop}[thm]{Proposition}

\newtheorem{Example}[thm]{Example}
\newtheorem{Number}[thm]{\!\!}
\newenvironment{defn}{\begin{Defn}\rm}{\end{Defn}}
\newenvironment{problem}{\begin{Problm}\rm}{\end{Problm}}
\newenvironment{example}{\begin{Example}\rm}{\end{Example}}
\newenvironment{rem}{\begin{Remark}\rm}{\end{Remark}}
\newenvironment{numba}{\begin{Number}\rm}{\end{Number}}
\newenvironment{proof}{{\noindent\bf Proof.}}%
                  {\nopagebreak\hspace*{\fill}$\Box$\medskip\medskip\par}   
\newcommand{\Punkt}{\nopagebreak\hspace*{\fill}$\Box$}
\newcommand{\wb}{\overline}
\newcommand{\ve}{\varepsilon}
\newcommand{\at}{\symbol{'100}}

\newcommand{\tensor}{\otimes}

\newcommand{\mto}{\mapsto}
\newcommand{\isom}{\cong}
\newcommand{\N}{{\mathbb N}}

\newcommand{\K}{{\mathbb K}}
\newcommand{\F}{{\mathbb F}}
\newcommand{\C}{{\mathbb C}}
\newcommand{\bL}{{\mathbb L}}
\newcommand{\R}{{\mathbb R}}
\newcommand{\Q}{{\mathbb Q}}
\newcommand{\Z}{{\mathbb Z}}

\newcommand{\cg}{{\mathfrak g}}

\DeclareMathOperator{\Aut}{Aut}

\newcommand{\one}{{\bf 1}}
\newcommand{\sub}{\subseteq}

\DeclareMathOperator{\id}{id}

\newcommand{\sbull}{{\scriptscriptstyle \bullet}}

\DeclareMathOperator{\tor}{tor}

\DeclareMathOperator{\car}{char}
\DeclareMathOperator{\op}{op}

\newcommand{\semid}{\mbox{$\times\!$\rule{.15 mm}{2.07 mm}}}
\begin{document}
\begin{center}
{\Large\bf
Contractible Lie Groups over
Local Fields}\\[6mm]
\renewcommand{\thefootnote}{\fnsymbol{footnote}}
{\bf Helge Gl\"{o}ckner\footnote{Research
supported by the German Research Foundation
(DFG), grant GL 357/6-1}}\vspace{4mm}
\end{center}
\renewcommand{\thefootnote}{\arabic{footnote}}
\setcounter{footnote}{0}
\begin{abstract}\vspace{1mm}
\hspace*{-7.2 mm}
Let $G$ be a Lie group over a local field
of characteristic~$p>0$
which admits a contractive automorphism
$\alpha\colon G\to G$
(i.e., $\alpha^n(x)\to 1$ as $n\to\infty$,
for each $x\in G$).
We show that
$G$ is a torsion group of finite exponent
and nilpotent.
We also obtain results
concerning the interplay between
contractive automorphisms of Lie groups
over local fields,
contractive automorphisms of their Lie algebras,
and positive gradations thereon.
Some of the results even extend
to Lie groups over arbitrary
complete ultrametric fields.\vspace{3mm}
\end{abstract}
{\footnotesize {\em Classification}:
22E20 (primary), 
20E15,           
20E36,           
26E30,           
37D10\\[2.5mm]   
{\em Key words}:
Lie group, local field, ultrametric field, positive characteristic,
contraction group, contractive automorphism,
contractible group,
torsion group,
positive gradation,
nilpotent group,
stable manifold,
Lie subgroup, composition series}\vspace{5mm}
\begin{center}
{\bf\Large Introduction}
\end{center}
A \emph{contraction group} is a pair
$(G,\alpha)$ of a topological group~$G$
and a (bicontinuous) automorphism
$\alpha\colon G\to G$ which
is \emph{contractive} in the sense
that $\alpha^n(x)\to 1$ as $n\to\infty$,
for each $x\in G$.
It is known from the work of Siebert~\cite{Sie}
that each locally compact contraction group
is a direct product
$G=G_0\times H$
of its identity component~$G_0$
and an $\alpha$-stable totally
disconnected group~$H$.
Siebert also showed that~$G_0$
is a simply connected, nilpotent
real Lie group.
Results concerning the totally disconnected part~$H$
were obtained in~\cite{SIM}.
It is a direct product
\[
H\;=\; H_{p_1}\times\cdots \times H_{p_n}\times \tor(H)
\]
of its subgroup $\tor(H)$ of torsion elements
and certain $\alpha$-stable $p$-adic Lie groups~$H_p$.
Thus $p$-adic contraction groups
are among the basic building blocks
of general contraction groups,
and it is therefore well motivated
to study these, and more generally
contraction groups which are
(finite-dimensional) Lie groups over local fields.
Essential structural information
concerning $p$-adic contraction groups
was obtained by
Wang~\cite{Wan}:
He showed that any such is a unipotent algebraic
group defined over~$\Q_p$
(and hence nilpotent).
The main goal of the current article
is to shed light on
contraction groups
which are Lie groups
over local fields of positive characteristic.\\[3.5mm]
{\bf Theorem A.}\,
\emph{Let $G$ be a $C^1$-Lie group
over a local field~$\K$ of characteristic $p>0$
which admits a contractive $C^1$-automorphism
$\alpha\colon G\to G$.
Then~$G$ is a torsion group
of finite exponent
and solvable.
Furthermore, there exists a series}
%
\begin{equation}\label{seon}
\one \, =\, G_0 \, \triangleleft \, G_1\, \triangleleft\,
\cdots\, \triangleleft\,  G_n\, =\, G
\end{equation}
\emph{of $\alpha$-stable, closed subgroups~$G_j$
such that the contraction group
$G_j/G_{j-1}$ is isomorphic to
$C_p^{(-\N)}\times C_p^{\N_0}$
with the right shift, for each $j\in \{1,\ldots, n\}$.}\\[3.5mm]
Here $C_p$ is the cyclic group of order~$p$
and $C_p^{(-\N)}\times C_p^{\N_0}$ the
restricted product
of all functions $f\colon \Z\to C_p$
such that $f(n)=1$ for $n$ below
some $n_0$, with the infinite power
$C_p^{\N_0}$ as a compact
open subgroup. The right shift~$\sigma$ is defined
via $\sigma(f)(n):=f(n-1)$,
and a morphism between
contraction groups $(G_1,\alpha_1)$ and
$(G_2,\alpha_2)$ is a continuous homomorphism
$\phi\colon G_1\to G_2$ such that
$\alpha_2 \circ \phi=\phi\circ\alpha_1$.
The series~(\ref{seon})
is a composition series of topological
$\langle\alpha\rangle$-groups (in the sense of~\cite{SIM}).\\[2.5mm]
We are mainly interested in $\K$-analytic ($C^\omega$-)
Lie groups and $\K$-analytic automorphisms,
but the preceding result holds just as well
for $C^1$-Lie groups and their automorphisms,
and has been formulated accordingly.
Recall that, while $C^k$-Lie groups and
analytic Lie groups coincide in the $p$-adic case~\cite{ANA},
for each local field of positive characteristic
there exist non-analytic smooth Lie groups
and $C^k$-Lie groups
which are not $C^{k+1}$,
for each $k\in \N$
(see \cite{NOA}).\\[2.5mm]
Our second main result
says that~$G$ is not only solvable,
but nilpotent,
at least under a slightly stronger
differentiability hypothesis ($k\geq 2$).
The conclusion even remains
valid for Lie groups over
complete ultrametric fields
which are not necessarily locally compact.\\[3.5mm]
{\bf Theorem~B.}\,
\emph{Given $k\in \N\cup\{\infty,\omega\}$
with $k\geq 2$,
let $G$ be a $C^k$-Lie group
over a complete ultrametric field $(\K,|.|)$
and $\alpha\colon G\to G$
be a contractive $C^k$-automorphism.
Then~$G$ is nilpotent.
Furthermore, there exists a central series}
%
\begin{equation}\label{censer}
\one \, =\, G_0 \, \triangleleft \, G_1\, \triangleleft\,
\cdots\, \triangleleft\,  G_m\, =\, G
\end{equation}
\emph{such that each $G_j$ is an $\alpha$-stable
$C^k$-Lie subgroup of~$G$.}\\[3.5mm]
While the proof of Theorem~A
(given in Section~\ref{sec1})
is based on the structure theory of totally
disconnected, locally compact contraction groups
from~\cite{SIM},
Theorem~B (proved in Section~\ref{sec3})
relies on entirely
different methods:
it uses the ultrametric stable manifolds
constructed in~\cite{STA}.\\[2.5mm]
To enable successful
application of the methods from~\cite{STA},
we first take a closer look
at the linearization
$L(\alpha)=T_1(\alpha)$
of~$\alpha$ around its fixed point~$1$
(see Section~\ref{secgr}).
For example, it is essential
for us that $L(\alpha)$
is a contractive automorphism
of $L(G)=T_1(G)$,
and that each eigenvalue
of~$L(\alpha)$ (in an algebraic closure) has absolute
value~$<1$
(as shown in~\cite{STA}).
Further results concerning
contractive Lie algebra automorphisms take
Siebert's treatment of the real case as a model.
He showed that each contractive automorphism
of a real Lie algebra~$\cg$
gives rise to a \emph{positive gradation}
on~$\cg$,
i.e., $\cg=\bigoplus_{r>0}\cg_r$
for vector subspaces
$\cg_r\sub \cg$ indexed by positive reals
such that $\cg_r=\{0\}$
for all but finitely many~$r$
and $[\cg_r,\cg_s]\sub \cg_{r+s}$
for all $r,s>0$.
Conversely, each positive gradation
yields contractive Lie algebra
automorphisms of~$\cg$~(see~\cite{Sie}).\\[2.5mm]
In the case of Lie algebras
over local fields,
the right class
of positive gradations
to look at are
\emph{$\N$-gradations},
i.e.\ positive gradations
$\cg=\bigoplus_{r>0}\cg_r$
such that $\cg_r\not=\{0\}$
implies $r\in \N$,
and thus $\cg=\bigoplus_{r\in \N}\cg_r$.
We show that
a Lie algebra~$\cg$ over a local field
admits an $\N$-gradation
if and only if it admits a contractive
Lie algebra automorphism (Proposition~\ref{charviagrad}).\\[2.5mm]
In Section~\ref{seclocglob},
we discuss the interplay between
contractive automorphisms of Lie groups
and Lie algebras.
In the real case, it is known that
each Lie group~$G$
admitting a contractive automorphism~$\alpha$
is simply connected,
and that $L(\alpha)$ is a contractive
Lie algebra automorphism.
Conversely, each contractive Lie algebra
automorphism of a real Lie algebra
integrates to a contractive
automorphism of the corresponding
simply connected real Lie group.
It is quite interesting
that, likewise,
we can always pass from
the Lie algebra level to the group
level in the case of complete ultrametric
fields of characteristic~$0$:\\[3.5mm]
{\bf Theorem~C.}\,
\emph{Let $(\K,|.|)$ be a complete
ultrametric field of characteristic~$0$,
$\cg$ be a Lie algebra over~$\K$
and $\beta\colon \cg\to\cg$
be a contractive Lie algebra automorphism.
Then there exists a $\K$-analytic
Lie group~$G$, unique up to isomorphism,
and a uniquely determined
$\K$-analytic contractive automorphism
$\alpha$ of~$G$ such that $L(\alpha)=\beta$.}\\[3.5mm]
In this case,
the appropriate substitute for
a simply connected group
is constructed with the help
of an HNN-extension.
Related
results are also obtained if $\car(\K)>0$,
but these are by necessity weaker.
For instance,
it may happen in positive
characteristic that two non-isomorphic
contraction groups give rise
to the same Lie algebra and the same
contractive Lie algebra automorphism.
An example for this phenomenon
(Example~\ref{ex221})
and examples illustrating various other
aspects of the theory are
compiled in Section~\ref{secexx}.\\[2.5mm]
Let us mention in closing that
results concerning contraction groups
also extend our knowledge of more general
automorphisms of Lie groups over local fields.
In fact, let~$G$ be a $C^k$-Lie group
over a local field, where $k\in \N\cup\{\infty,\omega\}$,
and $\alpha\colon G\to G$ be an automorphism
of~$C^k$-Lie groups.
Let $U_\alpha$ be the group of all
$x\in G$ such that $\alpha^n(x)\to 1$ as $n\to\infty$,
and $M_\alpha$ be the group of all $x\in G$
the two-sided orbit $\alpha^\Z(x)$ of which is
relatively compact in~$G$.
Then $(U_\alpha,\alpha|_{U_\alpha})$
and $(U_{\alpha^{-1}}, \alpha^{-1}|_{U_{\alpha^{-1}}})$
are contraction groups in the induced topology,
but they are also contraction groups
(with contractive $C^k$-automorphisms)
when equipped with suitable
immersed $C^k$-Lie subgroup structures
(see~\cite{STA}),\footnote{Making them the stable
manifold and unstable manifold of $\alpha$
around~$1$, respectively.}
which may correspond to properly
finer topologies.
Strongest results are available
if~$U_\alpha$ is closed.\footnote{This condition is
automatically
satisfied if $\car(\K)=0$ (see~\cite{Wan})
or if there exists an injective,
continuous homomorphism
from~$G$ to a general linear group~\cite{SPO}.
Various
characterizations of closedness of~$U_\alpha$
were given in~\cite{BaW}.}
Then all of $U_\alpha$, $U_{\alpha^{-1}}$
and $M_\alpha$ are $C^k$-Lie subgroups
of~$G$,
their complex product
$U_\alpha M_\alpha U_{\alpha^{-1}}$
is an open $\alpha$-stable
identity neighbourhood in~$G$,
and the product map
\[
U_\alpha \times M_\alpha\times
 U_{\alpha^{-1}}\to U_\alpha M_\alpha U_{\alpha^{-1}}
\]
is a $C^k$-diffeomorphism (see~\cite{Wan}
for the $p$-adic case, \cite{SPO}
for the general result).
The theorems of the current article
then apply to
$(U_\alpha,\alpha|_{U_\alpha})$
and $(U_{\alpha^{-1}}, \alpha^{-1}|_{U_{\alpha^{-1}}})$.
Some basic information on $M_\alpha$
can be drawn from~\cite{GW2} (cf.\ \cite{GW1}
and~\cite{Ra1} for the $p$-adic case):
$M_\alpha$ has small $\alpha$-stable
compact open subgroups. In contrast
to the case of contraction groups,
$M_\alpha$ need not have special
group-theoretic properties:
Choosing $\alpha=\id$, we get $G=M_\alpha$
and conclude that~$M_\alpha$
can be an arbitrary $C^k$-Lie group.\\[2.5mm]
Contraction groups of the form~$U_\alpha$
arise in many contexts:
In representation theory
in connection with the Mautner phenomenon
(see \cite[Chapter~II, Lemma~3.2]{Mar}
and (for the $p$-adic case) \cite{Wan});
in probability theory
on groups (see \cite{HaS}, \cite{Sie}
and (for the $p$-adic case)~\cite{DaS});
and in the structure theory
of totally disconnected, locally compact groups
developed in~\cite{Wil} (see \cite{BaW}).\\[3mm]
\emph{Acknowledgement.}
The author thanks George A. Willis
for useful discussions,
notably concerning the
examples in Section~\ref{secexx}.
\section{General conventions and facts}\label{sec0}
Complementing the definitions
already given in the Introduction,
we now fix additional notation and terminology.
%
%
\begin{numba}\label{convfields}
{\bf Conventions concerning valued fields.}
By a \emph{local field}, we mean
a totally disconnected,
non-discrete locally compact
topological field.
We fix an ultrametric absolute
value $|.|\colon \K\to [0,\infty[$ on~$\K$
defining its topology~\cite{Wei}.
A field~$\K$, equipped with
an ultrametric absolute value~$|.|$
which defines a non-discrete
topology on~$\K$ is called an
\emph{ultrametric field}; it is called \emph{complete}
if $\K$ is a complete metric space
with respect to the metric $(x,y)\mto |y-x|$.
Given a complete ultrametric field
$(\K,|.|)$ (e.g., a local field),
we fix an algebraic closure~$\wb{\K}$
of~$\K$ and use the same symbol, $|.|$,
for the unique extension of the given absolute value~$|.|$
to an absolute value on~$\wb{\K}$
(see \cite[Theorem~15.1]{Sch}).
If $(E,\|.\|)$ is a normed
space over a valued field~$(\K,|.|)$,
given $x\in E$ and $r>0$ we write
$B_r^E(x):=\{y\in E\colon \|y-x\|<r\}$
and $B_r:=B^E_r(0)$.
Given a continuous linear
map~$\alpha$ between normed spaces $(E,\|.\|_E)$
and $(F,\|.\|_F)$, its operator norm
is defined as
\[
\|\alpha\|_{\text{op}}\; :=\;\min\{r\in [0,\infty[\colon
(\forall x\in E)\; \|\alpha(x)\|_F\leq r\|x\|_E\}\,.
\]
\end{numba}
%
%
\begin{numba}\label{convcalc}
{\bf Differential calculus, manifolds and Lie groups.}
All manifolds, Lie groups and Lie algebras
considered in this article
are finite-dimensional.
Basic references for analytic
manifolds and analytic Lie groups over complete
ultrametric fields are \cite{Ser},
also \cite{Bo1} and~\cite{Bo2}.
We use the symbol ``$C^\omega$''
as a shorthand for ``analytic''
and agree that $n<\infty<\omega$
for all $n\in \N_0$,
where $\N=\{1,2,\ldots\}$ and $\N_0=\N\cup\{0\}$.
Let $E$ and~$F$ be (Hausdorff) topological
vector spaces over a non-discrete
topological field~$\K$
and $U\sub E$ be open.
Then $U^{[1]}:=\{(x,y,t)\in U\times E\times\K\colon
x+ty\in U\}$ is an open subset
of $E\times E\times \K$.
Following~\cite{Ber},
we say that $f$ is~$C^1$ if it is continuous and there
exists a (necessarily
unique) continuous map $f^{[1]}\colon
U^{[1]}\to F$ which extends the directional
difference quotient map, i.e.,
\[
f^{[1]}(x,y,t)\; =\;
\frac{f(x+ty)-f(x)}{t}
\]
for all $(x,y,t)\in U^{[1]}$ such that $t\not=0$.
Then $f'(x):=f^{[1]}(x,\sbull,0)\colon E\to F$
is a continuous linear map.
Inductively, $f$ is called $C^{k+1}$
for $k\in \N$ if~$f$ is $C^1$ and~$f^{[1]}$ is~$C^k$.
As usual, $f$ is called $C^\infty$
or smooth if $f$ is $C^k$ for all $k\in \N$.
If we want to stress~$\K$, we
shall also write~$C^k_\K$ in place of~$C^k$.\\[2.5mm]
In this article, we are only interested
in the case where $(\K,|.|)$
is a complete valued field and both~$E$
and~$F$ are finite-dimensional.
In the usual way, the above concept of $C^k$-map then
gives rise to a notion of (finite-dimensional) $C^k$-manifold
and a notion of (finite-dimensional) $C^k$-Lie group:
this is a group, equipped with
a $C^k$-manifold
structure which turns group
multiplication and inversion
into $C^k$-maps.
We let $L(G):=T_1(G)$
denote the tangent space at the identity element
$1\in G$ and set $L(f):=T_1(f)$
for a $C^k$-homomorphism $f\colon
G\to H$ between $C^k$-Lie groups.
If $k\geq 3$, then the Lie bracket
of left invariant vector fields
can be used in the usual
way to turn~$L(G)$ into
a Lie algebra, and $L(f)$ then
is a Lie algebra homomorphism.\\[2.5mm]
We mention that the $C^k$-maps used in this article
generalize the $C^k$-functions of a single variable
common in non-archimedian analysis
(as in~\cite{Sch}).
If $(\K,|.|)$ is a complete valued field,
then a map between open subsets
of finite-dimensional
vector spaces is~$C^1$
if and only if it is
strictly differentiable at each point
of its domain, in the sense of \cite[1.2.2]{Bo1}
(see \cite[Appendix C]{IM2};
for locally compact
fields, cf.\ also \cite[\S4]{IMP}).\footnote{This
fact enables us to use (and cite) various
results concerning $C^1$-maps
and $C^1$-Lie groups
from \cite{IMP}, \cite{IM2}
and \cite{ANA} also in the case
of non-locally compact,
complete
ultrametric fields,
which (strictly speaking) are formulated there
only in the locally compact case.
The proofs only use strict differentiability and
therefore carry over without
changes.}
In particular, such a map is totally differentiable
at each point.
For a survey of differential calculus
over topological fields
covering various aspects
of relevance for the current article,
the reader is referred to~\cite{ASP}.\\[2.5mm]
Because inverse- and implicit function theorems
are available for $C^k$-maps between
finite-dimensional vector spaces
over complete valued fields
(see \cite[notably Appendix C]{IM2}),\footnote{See
also \cite{IMP} for the cases where
$k\geq 2$
or the ground field is locally compact.}
we can define immersions
as in the analytic case~\cite{Ser},
with analogous properties.
If $k\in \N\cup\{\infty,\omega\}$,
$M$ is a $C^k$-manifold
and $N\sub M$ a $C^k$-manifold
such that the inclusion map $\iota\colon N\to M$ is an
immersion, we call~$N$ an \emph{immersed
$C^k$-submanifold} of~$M$;
if~$\iota$ is furthermore
a homeomorphism onto its image,
we call~$N$ a \emph{$C^k$-submanifold}.
Locally around each of its points,
$N$ then looks like a vector subspace
inside the modelling space of~$M$.
Given a $C^k$-Lie group~$G$,
a subgroup~$H$ equipped with a $C^k$-Lie
group structure making it a
$C^k$-submanifold of~$G$ (resp.,
an immersed $C^k$-submanifold)
is called a \emph{$C^k$-Lie subgroup}
(resp., an \emph{immersed $C^k$-Lie subgroup}).
In particular, every $C^k$-Lie subgroup of~$G$
is closed in~$G$.
\end{numba}
%
%
\begin{numba}\label{convaut}
{\bf Automorphisms and contraction groups.}
Given an automorphism $\alpha$ of a topological
group~$G$ and a subset $X\sub G$,
we say that~$X$ is \emph{$\alpha$-stable}
(resp., \emph{$\alpha$-invariant})
if $\alpha(X)=X$ (resp., $\alpha(X)\sub X$).
If we speak of $C^k$-isomorphisms
between $C^k$-Lie groups
or $C^k$-automorphism,
we assume that the inverse map
is $C^k$ as well.
A topological group (resp., $C^k$-Lie group)
$G$ is called \emph{contractible}
if it admits a contractive automorphism
(resp., a contractive $C^k$-automorphism).
Given a contraction group
$(G,\alpha)$, a series
$\one=G_0\triangleleft G_1\triangleleft\cdots
\triangleleft G_n=G$
of $\alpha$-stable, \emph{closed}
subgroups of~$G$
is called an \emph{$\langle\alpha\rangle$-series};
it is called a \emph{composition series}
if it does not admit a proper refinement
(see~\cite{SIM}).
%
\begin{defn}\label{unifcontr}
Let $(G,\alpha)$ be a contraction group.
\begin{itemize}
\item[(a)]
$\alpha$
is \emph{uniformly contractive}
$($or a \emph{uniform contraction}$)$
if each identity neighbourhood
of $G$ contains an $\alpha$-invariant
identity neighbourhood.
\item[(b)]
$\alpha$ is
\emph{compactly contractive}
if, for each compact set $K\sub G$
and identity neighbourhood $U\sub G$,
there is $n_0\in \N$ with
$\alpha^n(K)\sub U$~for~all~$n\geq n_0$.
\end{itemize}
\end{defn}
A simple compactness argument shows
that each uniformly contractive auto\-morphism
is compactly contractive.\\[2.5mm]
Although our main concern are contractive automorphisms
of Lie groups over local fields,
some of our results will also apply
to Lie groups
over non-locally compact, complete ultrametric fields,
like $\C_p$ and $\Q(\!(X)\!)$.\\[2.5mm]
Consider
a (finite-dimensional)
$C^k$-Lie group~$G$ over a complete ultrametric
field~$\K$.
Then~$G$ is complete
(see \cite[Proposition~2.1\,(a)]{ANA})\footnote{Recalling
footnote~3 if $k=1$ and $\K$ fails to be locally compact.}
and metrizable.
This implies that, automatically,
each contractive (bicontinuous) automorphism~$\alpha$
of such a Lie group
is uniformly and compactly contractive (cf.\ \cite[Lemma~1]{Si2}).
Since $G$ has arbitrarily small
open subgroups (see \cite[Proposition~2.1]{ANA}),
Siebert's construction in~\cite{Si2}
even produces small $\alpha$-invariant
identity neighbourhoods.
\end{numba}
We recall another useful fact,
the proof of which exploits that
contractive automorphisms of Lie groups
are uniformly contractive.
%
\begin{numba}\label{henceappl}
Let $(\K,|.|)$ be a complete ultrametric
field,
$G$ be a $C^1$-Lie group
and $\alpha\colon G\to G$ be a contractive
$C^1$-automorphism.
Then $\beta:=L(\alpha)$
is a contractive automorphism of~$L(G)$
and all eigenvalues of~$\beta$
in $\wb{\K}$
have absolute value $<1$
(see \cite{STA}).
Furthermore, there exists an ultrametric
norm $\|.\|$ on~$\cg$
such that $\|\beta\|_{\text{op}}<1$
holds for the corresponding operator
norm (see \cite{STA};
cf.\ also \cite[Lemma~3.3 and its proof]{SCA}
and \cite[Chapter~II, \S1]{Mar}).
\end{numba}
\section{Proof of Theorem~A}\label{sec1}
%
Given a contractive automorphism
$\alpha$ of a totally disconnected, locally
compact group~$G$, there exists
an $\alpha$-invariant, compact
open subgroup~$U$
such that $\alpha(U)\, \triangleleft \, U$
(see \cite[3.1 and Lemma~3.2]{Sie}),
whence $(\alpha^n(U))_{n\in \Z}$
is a filtration for~$G$ adapted to~$\alpha$ in the
sense of \cite[3.3]{Sie}.
This filtration can be used to compare
structures on~$G$ (or its subgroups).\footnote{This idea is
also the basis for the results in~\cite{NOA}.}
We shall use it in the proof
of Theorem~A to see that
the $\K$-Lie group structure
on~$G$ and the $p$-adic
Lie group structure on
a certain hypothetical
subgroup are incompatible.
The next lemma will be used
to relate the groups $\alpha^n(U)$ 
to balls in a local chart.
%
%
\begin{la}\label{prelem}
Let $(\K,|.|)$ be a complete
ultrametric field,
$k\in \N\cup\{\infty,\omega\}$,
$G$ be a $C^k$-Lie group
over~$\K$,
$S\sub G$ be an open subgroup
and $\alpha\colon S\to G$ be a $C^1$-homomorphism.
We set $\cg:=L(G)$,
$\beta:=L(\alpha)$
and assume that
$\beta$ is a linear isomorphism
and $\Theta:=\|\beta \|_{\op}<1$
for some ultrametric norm
$\|.\|$ on~$\cg$.
Then there exists
an $\alpha$-invariant
open identity neighbourhood~$U\sub S$
and a $C^k$-diffeomorphism
$\phi\colon U\to B_r\sub \cg$ for some $r>0$
with $\phi(1)=0$ and $T_1(\phi)=\id_\cg$,
such that the sets
$U_s:=\phi^{-1}(B_s)$
have the following properties:
\begin{itemize}
\item[\rm(a)]
$U_{\theta s}\sub \alpha(U_s)
\sub U_{\Theta s}$
for each $s\in \;]0,r]$,
where $\theta:=1/\|\beta^{-1}\|_{\op}$.
\item[\rm(b)]
$U_s$ is a subgroup of~$G$
for each $s\in \;]0,r]$,
and a normal subgroup of~$U_r$.
Also, $U_s/U_{\theta s}$ is abelian
for each $s\in \;]0,r]$,
and thus $\alpha(U_s) \triangleleft \, U_s$.
\end{itemize}
If $\car(\K)=0$ and $|p|<1$ for a prime~$p$,
one can also achieve:
\begin{itemize}
\item[\rm(c)]
$(U_s)^p=U_{|p|s}$ holds
for the set of $p$-th powers,
for each $s\in \;]0,r]$.
\end{itemize}
If $\K$ has characteristic $p>0$,
one can also achieve:
\begin{itemize}
\item[\rm(d)]
For each $\ve\in \;]0,1[$,
there exists $r_0\in \;]0,r]$
such that
$(U_s)^p\sub U_{\ve s}$ for each
$s\in \;]0,r_0]$.
\end{itemize}
\end{la}
\begin{proof}
Set $\theta:=\|\beta^{-1}\|^{-1}_{\op}$.
Then $0<\theta<\Theta<1$.
There exists a chart $\phi\colon U\to B_r$
for some open identity neighbourhood~$U\sub S$
and some $r>0$, such that
$\phi(1)=0$ and $T_1(\phi)=\id_\cg$.
After shrinking~$r$, we may assume
that $U_s:=\phi^{-1}(B_s)$ is
a subgroup of~$G$ for each
$s\in \;]0,r]$,
and that the remainder of~(b) as well
as (c)
resp.\ (d) hold
(see \cite[Proposition~2.1 (b), (f) and (i)]{ANA}).\footnote{The
formula $(U_s)^p=U_{|p|s}$ is shown
in~\cite{ANA} only if $|p|=p^{-1}$,
but the proof works as well
for arbitrary $|p|>0$.}
There exists $t\in \; ]0,r]$
such that $\alpha(U_t)\sub U_r$.
We can therefore define
a $C^k$-map
$\gamma\colon B_t\to B_r$
via $\gamma(x):=\phi(\alpha(\phi^{-1}(x)))$.
Our hypotheses ensure that~$\gamma$
is strictly differentiable at~$0$.
Now $\gamma'(0)$ being invertible,
the Ultrametric Inverse Function Theorem
\cite[Proposition~7.1]{IMP}
shows that, after shrinking~$t$ if necessary,
we have $\gamma(B_s)=\gamma'(0).B_s$
for each $s\in \;]0,t]$.
Since $B_{\theta s}\sub \gamma'(0).B_s\sub B_{\Theta s}$,
we deduce that~(a) holds after
replacing~$r$ with~$t$.
\end{proof}
{\bf Proof of Theorem~A.}
We recall from \cite[Theorem~B]{SIM}
that $G=D\times T$ internally
as a topological group,
where~$T$ is the subgroup of all torsion elements
and~$D$ the subgroup of all divisible elements.
Also by \cite[Theorem~B]{SIM},
$D$ is a direct product
$D_{p_1}\times\cdots \times D_{p_m}$,
where $p_k$ is a prime and $D_{p_k}$ a non-discrete
$p_k$-adic Lie group
for $k\in \{1, \ldots, m\}$.
Let $p$ be the characteristic of~$\K$.
Then~$G$ is locally pro-$p$,
i.e., it has a compact open subgroup
which is a pro-$p$-group
(see \cite[Proposition~2.1\,(h)]{ANA}).
Hence also each $D_{p_k}$
is locally pro-$p$ and
hence $p_k=p$ (cf.\ \cite[\S1.2]{Dix}),
whence~$D$ actually
is a $p$-adic Lie group.
To see that $D=\{1\}$,
let us assume that
$D\not=\{1\}$ and derive
a contradiction.
Being a non-trivial contraction group,
$D$ is then non-discrete
(see \cite[1.8\,(c)]{Sie}).\\[2.5mm]
Throughout the remainder
of the proof, the letters
(a)--(d) refer to the conditions
formulated in Lemma~\ref{prelem}.
Applying Lemma~\ref{prelem}
to~$G$ and~$\alpha$
(which is possible by {\bf\ref{henceappl}}),
we obtain $r>0$, $\theta:=1/\|L(\alpha)^{-1}\|_{\text{op}}\in \;]0,1[$
and an $\alpha$-invariant
compact open subgroup $U=U_r\sub G$
satisfying conditions
(a) and~(d). Then
%
\begin{equation}\label{stasta}
U_{\theta^ks}\; \sub \; \alpha^k(U_s)
\quad \mbox{for all $s\in \;]0,r]$ and $k\in \N$,}
\end{equation}
by a simple induction based on~(a).
By~(d), after shrinking~$r$,
we have
%
\begin{equation}\label{newe}
U^{p^k}\; \sub \; U_{\theta^k r}\quad
\mbox{for each $\,k\in \N$.}
\end{equation}
Since $\alpha|_D$
is a continuous
(and hence analytic)
contractive automorphism
of the $p$-adic Lie group~$D$,
applying
Lemma~\ref{prelem}
to $(\Q_p, |.|_p)$, $D$ and $\alpha|_D$
we get some $R>0$,
$\Theta:=\|L(\alpha|_D)\|_{\text{op}}\in \;]0,1[$,
a compact
open subgroup $V=V_R\sub D$
and subgroups $V_s\sub V_R$
satisfying analogues
of~(a) and~(c).
After shrinking~$R$,
we may assume that $V\sub U$.
Since $\alpha|_D$ is compactly
contractive,
there exists $N\in \N$ such that
$\alpha^N(U\cap D)\sub V$.
Choose $\ell\in \N$ so large that
$\ell\log_p(\Theta)<-1$
and set $\ve:=\theta^\ell$.
Since~$U$ satisfies~(d),
there exists $r_0\in \;]0,r]$
such that
\[
(U_s)^p\; \sub \; U_{\ve s}\;=\; U_{\theta^\ell s}
\quad
\mbox{for each
$s\in \;]0,r_0]$.}
\]
There is $M\in \N$ such that
$\theta^Mr<r_0$ and hence
$U^{p^M}\sub U_{\theta^Mr}\sub U_{r_0}$,
using~(\ref{newe}). 
Then
$U^{p^{k+M}}\sub U_{\theta^{k \ell} r_0}$
for each $k\in \N$,
by a trivial induction.
Here $U_{\theta^{k \ell} r_0}\sub \alpha^{k \ell}(U_{r_0})$,
by~(\ref{stasta}).
Thus
\begin{eqnarray*}
V_{p^{-k-M}R} & = & V^{p^{k+M}}
\; \sub \; U^{p^{k+M}}\cap D
\; \sub\; 
\alpha^{k \ell}(U_{r_0})\cap D\\
& \sub &  \alpha^{k\ell}(U)\cap D
\; =\;
\alpha^{k \ell}(U  \cap D)
\; \sub \;  \alpha^{k\ell-N}(V)
\; \sub \; V_{\Theta^{k\ell-N}R}
\end{eqnarray*}
for $k\in \N$ such that $k\ell \geq N$, using~(c) to obtain the first
equality.
As a consequence,
$p^{-k-M}R \leq p \Theta^{k\ell-N}R$
and hence
%
\begin{equation}\label{sta}
-k-M\; \leq \; 1+(k\ell-N) \log_p(\Theta)\,.
\end{equation}
Dividing both sides of~(\ref{sta})
by $k$ and letting $k\to\infty$,
we obtain the contradiction
$-1\leq \ell \log_p(\Theta)$.
Hence $D=\{1\}$ and thus $G=T$ is a torsion group.\\[2.5mm]
We now pick
a composition series~(\ref{seon})
of $\alpha$-stable closed subgroups
of~$G$ (as provided by \cite[Theorem~3.3]{SIM}).
Since $G_j/G_{j-1}$
is a torsion
group for each $j\in \{1,\ldots, n\}$,
the classification of the simple
contraction groups \cite[Theorem~A]{SIM}
shows that
$G_j/G_{j-1}\isom F_j^{(-\N)}\times F_j^{\N_0}$
with the right shift,
for some finite simple group~$F_j$.
Since~$G$ is locally pro-$p$,
so is $G_j/G_{j-1}$.
Hence~$F_j$ has to be a $p$-group,
entailing that $F_j\isom C_p$.
As a consequence, $x^{p^n}=1$
for each $x\in G$.
Each factor $G_j/G_{j-1}$ being abelian,
$G$ is solvable.\vspace{2mm}\Punkt
\section{Contractible Lie algebras and {\boldmath$\N$}-gradations}\label{secgr}
%
%
Let us call a Lie algebra~$\cg$ over a local
field~$\K$ \emph{contractible}
if there exists a contractive
Lie algebra automorphism $\alpha\colon \cg\to\cg$.
In this section, we prove
the following result:
%
\begin{prop}\label{charviagrad}
A Lie algebra $\cg$ over a local
field~$\K$ is contractible
if and only if it admits an
$\N$-gradation.
\end{prop}
The proof
is based on some facts
concerning automorphisms
of vector spaces over ultrametric fields,
which we now recall
(and which will be re-used later).
%
\begin{numba}\label{contrdec}
Let $E$ be a finite-dimensional
vector space over a
complete ultrametric field $(\K,|.|)$
and~$\alpha$ be a linear automorphism
of~$E$.
For each $r>0$, we let
$F_r$ be the sum of all generalized
eigenspaces of $\alpha\tensor_\K\id_{\wb{\K}}$
in $E\tensor_\K\wb{\K}$
to eigenvalues $\lambda\in \wb{\K}$
of absolute value $|\lambda|=r$.
By \cite[Chapter~II, \S1]{Mar},
$F_r$ is defined over~$\K$,
whence $F_r=E_r\tensor_\K \wb{\K}$
with $E_r:=F_r\cap E$.
Then
%
\begin{equation}\label{pregrad}
E \; =\; \bigoplus_{r>0}E_r\, .
\end{equation}
We call $r\in \; ]0,\infty[$
a \emph{characteristic value}
of~$\alpha$ if $E_r\not=\{0\}$,
and let $R(\alpha)$ be the set
of characteristic values.
There exists an ultrametric norm
on~$E$ such~that
%
\begin{equation}\label{googoo}
\|\alpha(v)\|\;=\; r\|v\|\quad
\mbox{for each $r\in R(\alpha)$ and $v\in E_r$}
\end{equation}
(see~\cite{STA}; cf.\ \cite[Lemma~3.3 and its proof]{SCA}).
Hence $\alpha$ is contractive
if and only if
$R(\alpha)\sub \;]0,1[$.
\end{numba}
{\bf Proof of Proposition~\ref{charviagrad}.}
Given a contractive Lie algebra
automorphism $\alpha\colon \cg\to\cg$,
let $\wb{\K}$ and~$|.|$ be as in~{\bf\ref{convfields}}.
There is $a>1$ such that
$|\K^\times|=\langle a\rangle\leq \R^\times$
(cf.\ \cite{Wei} or \cite[Corollary~12.2]{Sch}).
If $z\in \wb{\K}^\times$,
$\bL:=\K(z)$ and $d:=[\bL:\K]$ is the degree
of the field extension,
then $|z|=\sqrt[d]{|N_{\bL/\K}(z)|\,}\in \langle \sqrt[d]{a}\, \rangle$
using the norm $N_{\bL/\K}(z)$ (see
\cite[Theorem~9.8]{Jac}).
Therefore
%
\begin{eqnarray}\label{usesoon}
\log_a |\wb{\K}^\times|\; \leq\; \Q\,.
\end{eqnarray}
Applying the considerations from
{\bf\ref{contrdec}} to $E:=\cg$
and~$\alpha$, we obtain $R(\alpha)$,
spaces $F_r$
and vector subspaces $E_r\sub \cg$
with $\cg=\bigoplus_{r>0}E_r$.
Since $R(\alpha)\sub \;]0,1[$,
using (\ref{usesoon}),
we find $m\in \N$ such that
that $-m\log_a(R(\alpha))\sub \N$.
Hence
%
\begin{equation}\label{ngrd}
\cg\;=\; \bigoplus_{n\in \N} \cg_n
\end{equation}
with $\cg_n:=E_{a^{-n/m}}$.
Since $[F_r,F_s] \sub F_{rs}$
and hence $[E_r,E_s] \sub E_{rs}$
for all $r,s>0$ as a consequence
of Proposition~12\,(i)
in \cite[Chapter~7, \S1, no.\,4]{BoX},
it follows that~(\ref{ngrd})
is an $\N$-gradation.\\[2.5mm]
Conversely, assume that
$\cg=\bigoplus_{n\in \N}\cg_n$
is an $\N$-gradation.
Pick $\theta\in \K^\times$
such that $|\theta|<1$.
Then the unique $\K$-linear map
$\alpha\colon \cg\to\cg$
taking $x\in \cg_n$ to $\theta^nx$
is a contractive Lie algebra
automorphism of~$\cg$.\,\Punkt
\section{Contractible Lie groups are nilpotent}\label{sec3}
%
%
In this section, we prove Theorem~B.
The proof uses the stable
manifolds for ultrametric
dynamical systems constructed
in~\cite{STA}
by an adaptation
of Irwin's method (as
in~\cite{Ir1} and \cite{Wel}).\footnote{See also~\cite{IMP}
and~\cite{ASP}
for outlines of the main steps
of this construction.}
%
%
\begin{numba}\label{setsta}
Let $(\K,|.|)$ be a complete
ultrametric field
and $k\in \N\cup\{\infty,\omega\}$.
Let $M$ be a finite-dimensional
$C^k$-manifold over~$\K$,
$\alpha \colon M\to M$ be a $C^k$-diffeomorphism
and $z \in M$ be a fixed point of~$\alpha$.
Write $r_1<\cdots< r_n$ for
the characteristic
values of~$T_z(\alpha)$.
Given $a\in \; ]0,1[\, \setminus \{r_1,\ldots, r_n\}$,
let $W^s_a(M,z)$ be the set
of all $x\in M$
with the following property:
For some (and hence each)
chart $\phi\colon U\to V\sub T_z(M)$ of~$M$
around~$z$ such that $\phi(z)=0$
and $T_z(\phi)=\id_{T_z(M)}$,
and some (and hence each)
norm~$\|.\|$ on~$T_z(M)$,
there exists $n_0\in \N$
such that $\alpha^n(x)\in U$
for all integers $n\geq n_0$ and
%
\begin{equation}\label{asymp}
\lim_{n\to\infty} \frac{\|\phi(\alpha^n(x))\|}{a^n}\;=\; 0\,.
\end{equation}
\end{numba}
It is clear from the definition
that $W^s_a(M,z)$ is
an $\alpha$-stable subset of~$M$.
The following facts are proved
in~\cite{STA}:
%
\begin{numba}\label{ismfd}
\emph{For each $a\in \; ]0,1[\, \setminus \{r_1,\ldots, r_n\}$,
the set $W^s_a(M,z)$ is an immersed $C^k$-submanifold
of~$M$.} It is called the \emph{$a$-stable
manifold of~$M$ around~$z$}.
\end{numba}
%
\begin{numba}\label{ismfd2}
\emph{If $\{r_1,\ldots, r_n\}\sub \;]0,1]$,
then $W^s_a(M,z)$ is a $C^k$-submanifold of~$M$,
for each $a\in \; ]0,1[\, \setminus \{r_1,\ldots, r_n\}$.}
\end{numba}
%
\begin{numba}\label{nomatt}
If $0<a<b<1$
and $[a,b]\cap \{r_1,\ldots, r_n\}=\emptyset$,
then
$W^s_a(M,z)= W^s_b(M,z)$.
\end{numba}
%
%
\begin{numba}\label{whentriv}
\emph{If $a\in \; ]0,r_1[$, then
$W^s_a(M,z)=\{z\}$.}\,\Punkt
\end{numba}
%
%
\begin{prop}\label{nice2know}
Let
$k\in \N\cup\{\infty,\omega\}$
and $(\K,|.|)$
be a complete ultrametric field.
Let $G$ be a $C^k$-Lie group over~$\K$
and $\alpha\colon G\to G$ be
a $C^k$-automorphism.
Assume that $a\in \;]0,1[$
is not a characteristic value
of~$L(\alpha)$.
Then the $a$-stable
manifold
$W_a^s(G,1)$ is an immersed
$C^k_\K$-Lie subgroup of~$G$.
If $\alpha$ is a contractive
$C^k$-automorphism, then
$W_a^s(G,1)$ is a $C^k_\K$-Lie subgroup of~$G$.
\end{prop}
\begin{proof}
We first show that $H:=W^s_a(G,1)$
is a subgroup of~$G$.
To this end, we pick a chart $\phi\colon
U\to V\sub T_1(G)=L(G)$
as in {\bf \ref{setsta}}
and an ultrametric norm~$\|.\|$ on
$L(G)$; we use the same symbol, $\|.\|$,
for the corresponding maximum norm on
$L(G)\times
L(G)$.
After shrinking~$U$,
we may assume that~$U$ is a subgroup
of~$G$ and give~$V$ the group
structure making~$\phi$ a homomorphism
(see \cite[Proposition~2.1]{ANA}).
After shrinking~$U$ further,
we may assume that
%
\begin{equation}\label{use2}
\|xy^{-1}-(x-y)\|\leq \|(x,y)\|\quad \mbox{for all
$x,y\in V$,}
\end{equation}
because $h\colon V\times V\to V$, $(x,y)\mto xy^{-1}$
is totally differentiable at $(0,0)$
with $h'(0,0)\colon L(G)\times L(G)\to L(G)$,
$(u,v)\mto u-v$.\\[2.5mm]
If $x,y\in H$,
there exists $n_0\in \N$ such that
$\alpha^n(x), \alpha^n(y)\in U$
for all $n\geq n_0$
and $\|\phi(\alpha^n(x))\|/a^n,\|\phi(\alpha^n(y))\|/a^n\to 0$.
Then $\alpha^n(xy^{-1})=\alpha^n(x)\alpha^n(y)^{-1}
\in UU^{-1}=U$ and
\begin{eqnarray*}
\lefteqn{\|\phi(\alpha^n(xy^{-1}))\|/a^n\;=\;
\|\phi(\alpha^n(x))\phi(\alpha^n(y))^{-1}\|/a^n}\\
& \leq & \max\{\|\phi(\alpha^n(x))-\phi(\alpha^n(y))\|,
\|(\phi(\alpha^n(x)),\phi(\alpha^n(y))^{-1})\|\}/a^n\\
& = & \max\{\|\phi(\alpha^n(x))\|/a^n,\|\phi(\alpha^n(y))\|/a^n\}
\; \to \; 0
\end{eqnarray*}
as $n\to\infty$,
showing that $xy^{-1}\in H$. Hence~$H$ is a subgroup indeed.\\[2.5mm]
To see that~$H$ is an immersed $C^k$-Lie subgroup,
we recall from the construction of $a$-stable
manifolds that $\alpha|_H\colon H\to H$
is a $C^k$-diffeomorphism
and that there
is an
$\alpha$-invariant
open subset $\Gamma \sub H$
(a ``local $a$-stable manifold'')
such that
$H=\bigcup_{n\in \N_0}\alpha^{-n}(\Gamma)$
and~$\Gamma$ is a submanifold of~$G$.
For a suitable choice of the chart
$\phi\colon U\to V$ and
ultrametric norm $\|.\|$ on $L(G)$
above,
one has $V=B_r \sub L(G)$
for some $r>0$ in the construction
and the set~$\Gamma$ consists of
all $x\in U$
such that
\begin{itemize}
\item[($\diamondsuit$)]
$\alpha^n(x)\in U$
for all $n\in \N_0$,
$\|\phi(\alpha^n(x))\|\leq a^n$
for all $n\in \N_0$,
and $\|\phi(\alpha^n(x))\|/a^n\to 0$
as $n\to\infty$
\end{itemize}
(see~\cite{STA}).
After shrinking~$U$ (and~$r$)
if necessary, we may assume
that~$U$ is a subgroup of~$G$
and the estimates~(\ref{use2}) hold.
Given $x,y\in \Gamma$,
we can use (\ref{use2}) as above
to see that also $xy^{-1}$
satisfies the conditions~($\diamondsuit$),
and hence $xy^{-1}\in \Gamma$.
Thus~$\Gamma$ is a subgroup
of~$G$ and hence a $C^k$-Lie subgroup.
As a consequence, also
$\alpha^{-n}(\Gamma)$ is a $C^k$-Lie
subgroup of~$G$.
Since each $\alpha^{-n}(\Gamma)$
is an open $C^k$-submanifold
of the $a$-stable manifold~$H$
and $H=\bigcup_{n\in \N_0}\alpha^{-n}(\Gamma)$,
it follows that the group
operations of~$H$ are $C^k$
on an open cover and hence~$C^k$.
Thus~$H$ is an immersed $C^k$-Lie subgroup
of~$G$.\\[2.5mm]
If $\alpha$ is contractive,
then $R(\alpha)\sub \;]0,1[$
(see {\bf\ref{henceappl}}).
Hence $H=W^s_a(G,1)$
is a $C^k$-submanifold
of~$G$ (by {\bf\ref{ismfd2}})
and therefore a $C^k$-Lie subgroup.
\end{proof}
Given subsets $X,Y$ of a group~$G$,
we set $[X,Y]:=\{xyx^{-1}y^{-1}\colon
x\in X, y\in Y\}$.
%
%
\begin{la}\label{bettappr}
Let $G$ be a $C^2$-Lie group
over a complete ultrametric field\linebreak
$(\K,|.|)$.
Let $\alpha\colon G\to G$ be
a $C^2$-automorphism
and assume that $a,b\in \;]0,1[$
as well as $ab$ are not characteristic values
of~$L(\alpha)$.
Then
\[
[W_a^s(G,1), W_b^s(G,1)]\;\sub\;
W_{ab}^s(G,1)\,.
\]
\end{la}
\begin{proof}
We pick a chart $\phi\colon U\to V\sub L(G)$
of~$G$ around~$1$
such that $\phi(1)=0$
and $T_1(\phi)=\id_{L(G)}$.
After shrinking~$U$ further, we may
assume that~$U$ is a subgroup
of~$G$. We give~$V$ the group
structure making~$\phi$ an isomorphism.
Then~$V$ is a $C^2$-Lie group.
The commutator map
$f\colon V\times V\to V$,
$f(x,y)=xyx^{-1}y^{-1}$
is $C^2$ and satisfies $f(x,0)=f(0,y)=0$.
Hence, by \cite[Lemma~1.7]{ANA},
after shrinking~$V$ there exists $C>0$
such that
\[
\|f(x,y)\|\;\leq\; C\, \|x\|\cdot\|y\|\quad
\mbox{for all $x,y\in V$.}
\]
Given $x\in W^s_a(G,1)$ and $y\in W^s_b(G,1)$,
there exists $n_0\in \N$
such that $\alpha^n(x),\alpha^n(y)\in U$ for all $n\geq n_0$.
Then
\begin{eqnarray*}
\frac{\|\phi(\alpha^n(xyx^{-1}y^{-1}))\|}{(ab)^n}
& = &
\frac{\|f(\phi(\alpha^n(x)),\phi(\alpha^n(y)))\|}{(ab)^n}\\
& \leq &
C\, \frac{\|\phi(\alpha^n(x))\|}{a^n}
\, \frac{\|\phi(\alpha^n(y))\|}{b^n}
\; \to \; 0
\end{eqnarray*}
as $n\to\infty$ (see (\ref{asymp})),
and thus $xyx^{-1}y^{-1}\in W^s_{ab}(G,1)$.
\end{proof}
{\bf Proof of Theorem~B.}
We may assume that $G\not=\{1\}$.
Since $\alpha$ is contractive,
it follows that $R(L(\alpha))\sub \; ]0,1[$
(see {\bf\ref{henceappl}}).
Let $0<r_1<\cdots< r_m<1$
be the characteristic
values of~$L(\alpha)$.
Pick $a_m\in \;]r_m,1[$.
Next,
for $j\in \{1,\ldots, m-1\}$,
pick $a_j\in \;]r_j,r_{j+1}[$
so small that $a_ja_m<r_j$,
and such that $a_ja_i\not\in \{r_1,\ldots, r_m\}$ for all
$i\geq j$.
Set $a_0:=a_1a_n$.
By Proposition~\ref{nice2know},
$G_j:=W^s_{a_j}(G,1)$
is a $C^k$-Lie subgroup
of~$G$, for $j\in \{0,1,\ldots, m\}$.
Furthermore, each~$G_j$ is $\alpha$-stable,
and $G_0=\{1\}$, by {\bf\ref{whentriv}}.
Also, $G_m=G$ (cf.\ \cite{STA}).
By Lemma~\ref{bettappr}, we have
\[
[G,G_j]\; =\; [W^s_{a_m}(G,1),W^s_{a_j}(G,1)]
\; \sub \; W^s_{a_ma_j}(G,1)
\; \sub \; W^s_{a_{j-1}}(G,1)\;=\; G_{j-1}
\]
for $j\in \{1,\ldots, m\}$.
Hence each $G_j$ is normal in~$G$
and $G_j/G_{j-1}$ is contained in the
centre of~$G/G_{j-1}$,
showing that
$\one=G_0\triangleleft G_1\triangleleft\cdots
\triangleleft G_m=G$ is a central series.
In particular, $G$ is nilpotent (see \cite[p.\,122]{Rob}).\,\vspace{2mm}\Punkt

\noindent
If~$\alpha$ merely is a contractive $C^1$-automorphism
in Theorem~B, the preceding proof
still provides a central
series of~$C^1$-Lie subgroups
(it is only essential that the commutator
map~$f$ is~$C^2$).
\section{From contractible Lie algebras
to\\
contractible Lie groups}\label{seclocglob}
%
%
In this section,
we discuss
the passage from contractive Lie algebra
automorphisms to contractive
Lie group automorphisms.
We begin with a result
which subsumes
Theorem~C from the introduction
(when specialized to characteristic~$0$).
Afterwards, we work towards
a category-theoretic refinement
of Theorem~C: an
equivalence between
the category of analytic Lie contraction groups
and the category of Lie algebra-contraction pairs.
%
%
\begin{prop}\label{gradthm}
Let $(\K,|.|)$ be a complete
ultrametric field,
$k\in \N\cup \{\infty,\omega\}$
such that $k\geq 3$,
$\cg$ be a Lie algebra over~$\K$
and $\beta\colon \cg\to\cg$
be a contractive Lie algebra
automorphism.
We make the following
assumption:
\begin{itemize}
\item[$(*)$]
There exists a $C^k$-Lie group~$V$
with $L(V)=\cg$,
an open subgroup~$U\sub V$
and a
$C^k$-homomorphism
$\gamma \colon U\to V$
such that $L(\gamma)=\beta$.
\end{itemize}
Then there exists
a $C^k$-Lie group~$G$
and a contractive $C^k$-auto\-mor\-phism
$\alpha\colon G\to G$ such that
$L(\alpha)=\beta$.
If $\car(\K)=0$,
then condition~$(*)$
is automatically satisfied.
If $\car(\K)=0$ and $k=\omega$,
then furthermore $(G,\alpha)$ is
unique up to isomorphism.
\end{prop}
\begin{proof}
By {\bf\ref{contrdec}},
there exists an ultrametric
norm~$\|.\|$ on $\cg$ such that
$\Theta:=\|\beta\|_{\op}<1$.
Hence Lemma~\ref{prelem}
applies to $\gamma\colon U\to V$,
and ensures that
after shrinking~$U$,
there is $r>0$ and a chart
$\phi\colon U\to B_r\sub \cg$
with $\phi(1)=0$
and $T_1(\phi)=\id_\cg$
such that $\gamma(U_s)\sub U_{\Theta s}$
for each $s\in \;]0,r]$,
with $U_s:=\phi^{-1}(B_s)$.
In particular, this implies that
$\gamma(U)\sub U$
and that $\{\gamma^n(U)\colon
n\in \N_0\}$ is a basis of $\gamma$-stable
identity neighbourhoods in~$U$
(since $\gamma^n(U)\sub U_{\Theta^n r}$).
Hence $\gamma|_U\colon U\to U$
is contractive.
Since $T_1(\gamma)=\beta$ is invertible,
after shrinking~$r$ if necessary
we may assume
that also $\gamma(U)$ is open
in~$U$ and $\gamma\colon U\to \gamma(U)$
is a $C^k$-diffeomorphism
(using the Inverse Function Theorem
\cite[Theorem~5.1]{IM2}, resp.,
\cite[Part~II, Chapter~III, \S9, Theorem~2]{Ser}).\\[2.5mm]
The group $V$ together with the isomorphism
$\gamma|_U \colon U\to \gamma(U)$ between its subgroups
gives rise to an HNN-extension~$W$.
This is a group~$W$ which contains~$V$
as a subgroup and has an element $w \in W$
such that
\[
wxw^{-1}\;=\; \gamma(x)\quad\mbox{for all $\,x\in U$}
\]
(see, e.g., 6.4.5 and the remarks following it
in \cite{Rob}).
Consider the inner automorphism
$\alpha \colon W\to W$, $\alpha(x):=wxw^{-1}$.
Then $\alpha(U)\sub U$
and $\alpha|_U$ is $C^k$.
Furthermore, $\alpha^{-1}|_{\alpha(U)}$
is $C^k$ on the open identity neighbourhood
$\alpha(U) \sub U$.
Since~$U$ is a $C^k$-Lie group,
standard arguments now provide
a unique $C^k$-Lie group
structure on
the subgroup $H:=\langle U\cup\{w\}\rangle\leq W$
generated by~$U$ and~$w$ which makes~$U$
an open $C^k$-submanifold of~$H$.
Since $\alpha(U)=\gamma(U)\sub U$,
it follows that $G:=\bigcup_{n\in \N}\alpha^{-n}(U)\sub H$
is an open subgroup of~$H$
and $\alpha|_G\colon G\to G$
a contractive automorphism
with $L(\alpha|_G)=L(\gamma)=\beta$.\\[2.5mm]
If $\car(\K)=0$,
we choose an ultrametric
norm $\|.\|$ on~$\cg$ such that
$\|\beta\|_{\op}<1$.
For some $t>0$,
the Baker-Campbell-Hausdorff (BCH-)
series
then converges on $B_t\times B_t$
(where $B_t:=B_t^\cg(0)$)
to a function taking its values in~$B_t$,
and making $U:=B_t$
a $\K$-analytic Lie group
with $L(U)=\cg$
(see Lemma~3 in \cite[Chapter~3, \S4, no.\,2]{Bo2}).
Then $V:=U$ together with
$\gamma:=\beta|_U$ satisfies condition~($*$).
The uniqueness assertion
is covered by Lemma~\ref{integr}
below.
\end{proof}
If $\car(\K)>0$, then
an analytic Lie contraction group
$(G,\alpha)$ need not be determined
by $(L(G),L(\alpha))$ (see Example~\ref{ex221}).
%
%
\begin{rem}\label{prefin}
The preceding proof
shows that if~($*$) holds,
then after shrinking~$U$ we can assume that
\begin{itemize}
\item[($**$)]
There is a $C^k$-Lie group~$U$
and a $C^k$-homomorphism $\gamma\colon U\to U$
with open image
such that
$L(\gamma)=\beta$
and $\gamma\colon U\to \gamma(U)$ is
a $C^k$-diffeomor\-phism with
$\{\gamma^n(U)\colon n\in \N_0\}$
a basis of identity neighbourhoods in~$U$.
\end{itemize}
In this case, $(G,\alpha)$ can be chosen
such that $U$ is an open subgroup of~$G$
and $\alpha|_U=\gamma$.
\end{rem}
The following lemma is a variant of \cite[Proposition~2.2]{Wan}.
%
%
\begin{la}\label{extloc}
Let $(G_1,\alpha_1)$
be a contraction group,
$U\sub G_1$ be an $\alpha_1$-invariant
open subgroup,
$G_2$ be a group,
$\alpha_2$
be an automorphism of~$G_2$
and $g\colon U\to G_2$ be a homomorphism
such that $\alpha_2\circ g=g\circ \alpha_1|_U$.
Then~$g$ extends uniquely
to a homomorphism $h \colon G_1\to G_2$
such that $\alpha_2\circ h =h \circ \alpha_1$.
If also $(G_2,\alpha_2)$
is a contraction group,
$g(U)$ open
and $g \colon U\to g(U)$
a homeomorphism,
then~$h$ is an isomorphism
of topological groups.
\end{la}
\begin{proof}
First assertion: The hypotheses ensure
$U\sub\alpha_1^{-1}(U)\sub \alpha^{-2}_1(U)\sub\cdots$
and $G_1=\bigcup_{n\in \N_0}\alpha^{-n}_1(U)$.
Given $x\in \alpha^{-n}_1(U)$,
we set $h(x):=\alpha^{-n}_2(g(\alpha^n_1(x)))$.
It is easy to see that~$h$
is well defined
and has
the desired properties.\\[2.5mm]
Second assertion:
Since $V:=g(U)$
satisfies $\alpha_2(V)\sub V$
and $\alpha_1\circ g^{-1}=g^{-1}\circ \alpha_2|_V$,
the first assertion yields
a homomorphism $k \colon G_2\to G_1$
such that $k|_V=g^{-1}$ and $\alpha_1\circ k=k\circ\alpha_2$.
Then $h\circ k=\id_{G_2}$ and $k\circ h=\id_{G_1}$,  
by the uniqueness assertion.
\end{proof}
%
%
\begin{la}\label{integr}
Let $(\K,|.|)$ be a complete ultrametric
field of characteristic~$0$.
Let $G_j$ be an analytic
Lie group over~$\K$
and $\alpha_j$ be a
contractive,
analytic automorphism
of~$G_j$, for $j\in \{1,2\}$.
Let $f\colon L(G_1)\to L(G_2)$
be a Lie algebra homomorphism
with $f \circ L(\alpha_1)=L(\alpha_2)\circ
f$. Then there is a
unique analytic homomorphism
$f^\wedge \colon G_1\to G_2$
such that $L(f^\wedge )= f$ and
$\alpha_2\circ f^\wedge =f^\wedge \circ
\alpha_1$.
If~$f$ is a Lie algebra
isomorphism, then~$f^\wedge$
is an analytic isomorphism.
\end{la}
\begin{proof}
By Lemma~\ref{prelem} and {\bf\ref{henceappl}},
$G_1$ has arbitrarily small
$\alpha_1$-invariant open subgroups~$U$.
By Theorem~1\,(i)
in \cite[Chapter~3, \S4, no.\,1]{Bo2},
after choosing~$U$ small enough
there exists an analytic homomorphism
$g\colon U\to G_2$ such that
$L(g)=f$. Since $L(g\circ \alpha_1|_U)=L(g)\circ L(\alpha_1)=
f\circ L(\alpha_1)=L(\alpha_2)\circ f=L(\alpha_2\circ g)$,
part\,(ii) of the theorem
just cited shows that $g\circ \alpha_1|_U=\alpha_2\circ g$,
after choosing~$U$ even smaller if necessary.
Now Lemma~\ref{extloc}
provides a unique homomorphism
$f^\wedge \colon G_1\to G_2$ such that
$\alpha_2\circ f^\wedge =f^\wedge \circ \alpha_1$
and $f^\wedge|_U=g$.
Since~$g$ is analytic, so is~$f^\wedge$,
and $L(f^\wedge)=L(g)=f$.\\[2.5mm]
If also $f^*\colon G_1\to G_2$ is an analytic homomorphism
with the desired properties,
then $f^*|_V=f^\wedge|_V$
for a sufficiently small
$\alpha$-stable open subgroup $V\sub G_1$
(which a priori might be smaller than~$U$
just used), because $L(f^*)=L(f^\wedge)$.
Hence $f^*=f^\wedge$ by uniqueness
in Lemma~\ref{extloc}.
To complete the proof,
note that $g(U)$ is open and $g\colon U\to g(U)$
is an analytic diffeomorphism
in the preceding construction
if we choose~$U$ sufficiently small,
and hence~$f^\wedge$ is an isomorphism
by Lemma~\ref{extloc}.
\end{proof}
\begin{defn}
Let $(\K,|.|)$ be a complete ultrametric
field of characteristic~$0$.
We then obtain categories
$\mathbb{CLG}_\K$
and $\mathbb{CLA}_\K$, as follows:
\begin{itemize}
\item
The objects of
$\mathbb{CLG}_\K$
are pairs $(G,\alpha)$,
where $G$ is an analytic Lie group
over~$\K$ and $\alpha\colon G\to G$
a contractive,
analytic automorphism.
A morphism $(G_1,\alpha_1)\to (G_2,\alpha_2)$
in $\mathbb{CLG}_\K$
is an analytic homomorphism
$f\colon G_1\to G_2$ such that $\alpha_2\circ f=f\circ\alpha_1$.
\item
The objects of
$\mathbb{CLA}_\K$
are pairs $(\cg,\beta)$,
where $\cg$ is a Lie algebra over~$\K$
and $\beta\colon \cg\to \cg$
a contractive Lie algebra automorphism.
A morphism $(\cg_1 ,\beta_1)\to (\cg_2 ,\beta_2)$
is a Lie algebra homomorphism
$f\colon \cg_1 \to \cg_2$ such that $\beta_2\circ f=f\circ\beta_1$.
\end{itemize}
\end{defn}
We now show:
\begin{thm}
The categories 
$\mathbb{CLG}_\K$
and $\mathbb{CLA}_\K$
are equivalent.
\end{thm}
\begin{proof}
It is clear that a covariant functor
$P \colon \mathbb{CLG}_\K
\to \mathbb{CLA}_\K$
can be defined via
$P (G,\alpha):=(L(G),L(\alpha))$
on objects and $P(f):=L(f)$ on morphisms
(cf.\ {\bf\ref{henceappl}}).
We now define a covariant
functor $Q\colon \mathbb{CLA}_\K
\to \mathbb{CLG}_\K$.
Given an object $x=(\cg,\beta)$,
we let $Q(x):=(G,\alpha)$ be an
analytic Lie contraction group
such that $L(G)=\cg$ and $L(\alpha)=\beta$,
as constructed in Proposition~\ref{gradthm}.
More precisely, we identify
$\cg$ with~$L(G)$ here by means of a fixed
Lie algebra isomorphism
%
\begin{equation}\label{natiso1}
\phi_x\colon \cg\to L(G)\, ,
\end{equation}
and require that
$L(\alpha)=\phi_x \circ \beta\circ \phi_x^{-1}$.
Given objects $x_j=(\cg_j,\beta_j)$
for $j\in \{1,2\}$
and a morphism $f\colon (\cg_1,\beta_1)\to
(\cg_2,\beta_2)$,
we define a morphism $Q(x_1)\to Q(x_2)$
via
$Q(f):= (\phi_{x_2} \circ f \circ \phi_{x_1}^{-1})^\wedge$,
using notation as in Lemma~\ref{integr}.
Then it is easy to see that~$Q$ is a functor
and that $\phi$ is a natural isomorphism
from $\id$
to $P\circ Q$ (in the sense of \cite[p.\,16]{Mac}).
Furthermore, a natural isomorphism~$\psi$
from $\id$ to $Q\circ P$ can be defined
as follows:\\[2.5mm]
Given an object $y=(G,\alpha)$
in $\mathbb{CLG}_\K$,
we have $P(y)=(L(G),L(\alpha))=:x$
and $Q(P(y))=(G_x,\alpha_x)$,
where $L(\alpha_x)=\phi_x \circ L(\alpha)\circ \phi_x^{-1}$.
By Lemma~\ref{integr},
there exists a unique isomorphism
$\psi_y:=(\phi_x)^\wedge\colon G\to
G_x$ such that $\alpha_x\circ \psi_y=\psi_y\circ\alpha$
and
$L(\psi_y)= \phi_x \colon L(G)\to L(G_x)$.
The naturality is easy to check.\\[2.5mm]
We have shown that the functors
$P$ and~$Q$ define an equivalence
of categories between $\mathbb{CLG}_\K$
and $\mathbb{CLA}_\K$ (in the sense
of \cite[p.\,18]{Mac}).
\end{proof}
\section{Examples and open problems}\label{secexx}
%
%
We start with examples related to Theorem~A.
\begin{example}
Let $\K$ be a local field of positive
characteristic. Since $\K^\times$
is not a torsion group, it does not admit
a contractive $C^1$-automorphism,
by Theorem~A.
In fact, $\K^\times$ does not even
admit a contractive \emph{bicontinuous}
automorphism. To see
this, we assume the existence
of such an automorphism~$\alpha$
and derive a contradiction.
We pick an element $0\not=x\in \K^\times$
such that $|x|\not=1$.
Then $D:=\langle x\rangle$ is an infinite
cyclic group and
discrete in the topology induced by~$\K^\times$.
We let $U\sub \K^\times$ be a compact open subgroup.
Then $\alpha^n(x)\in U$ for some~$n$
and hence $\alpha^n(D)$ is an infinite discrete
subgroup of the compact group~$U$,
which is absurd.
\end{example}
The following example shows
that it is in general not possible
to choose all of the groups~$G_j$ in
a composition series (\ref{seon})
as $C^k_\K$-Lie subgroups of~$G$.
%
%
%
\begin{example}\label{exwithl}
Let $\K:=\F_p(\!(X)\!)$ be the field of formal Laurent
series over a finite field~$\F_p$ with~$p$ elements,
$G:=(\K,+)$ and $\alpha\colon G\to G$, $z\mto X^2z$.
Then $G_1:=\F_p^{(-2\N)} \times \F_p^{2\N_0}\sub
\F_p^{(-\N)}\times \F_p^{\N_0}=\K$
is an $\alpha$-stable closed
subgroup of~$G$.
Furthermore $G_1\isom \F_p^{(-\N)}\times \F_p^{\N_0}$
via $(x_n)_{n\in 2\Z}\mto (x_{2n})_{n\in \Z}$
and $G/G_1\isom \F^{(-2\N+1)}\times \F^{2\N_0+1}$.
As both contraction groups are isomorphic
to $C_p^{(-\N)}\times C_p^{\N_0}$ with the right shift,
they are simple contraction groups
and hence
\[
\one\, \triangleleft \, G_1\, \triangleleft \, G
\]
is a composition series
of closed $\alpha$-stable subgroups
of~$G$.
Let $\one\triangleleft H_1\triangleleft G$
be any such composition series.
We now show that~$H_1$ is not a Lie subgroup.
In fact, $H_1$ is a non-discrete, proper
subgroup of~$G$.
Hence, if $H_1$ would be a Lie subgroup
of~$G$, it would be $1$-dimensional
and hence open in the $1$-dimensional
Lie group~$G$.
Then $G=\bigcup_{n\in \N_0}\alpha^{-n}(H_1)=H_1$,
which is absurd.
\end{example}
In other cases, the groups $G_j$
can be chosen only as $C^k_\bL$-Lie subgroups
for some subfield $\bL\sub \K$.
\begin{example}
Let $\K:=\F_{p^2}(\!(X)\!)$,
$G:=(\K,+)$ and $\alpha\colon G\to G$, $z\mto Xz$.
Then $G_1:=\F_p(\!(X)\!)$
is an $\alpha$-stable closed
subgroup~$G$.
Since both~$G_1$ and $G/G_1$
are isomorphic as contraction groups to
$C_p^{(-\N)}\times C_p^{\N_0}$ with the right shift,
they are simple contraction groups
and hence
$\one\triangleleft G_1\triangleleft G$
is a composition series
of closed $\alpha$-stable subgroups
of~$G$. Here $G_1$ is a $C^\omega_\bL$-Lie subgroup
over $\bL:=\F_p(\!(X)\!)$.
However, neither~$G_1$
nor any other group~$H_1$ in a composition series
$\one\triangleleft H_1\triangleleft G$
of $\alpha$-stable closed subgroups
can be a $C^1_\K$-Lie subgroup,
because~$G$ is $1$-dimensional
over~$\K$, enabling
us to argue as in Example~\ref{exwithl}.
\end{example}
Of course, instead of a composition series
of closed $\alpha$-stable
subgroups,
in the situation of Theorem~A\hspace*{.2mm}
we can consider
a properly ascending series
\[
\one\, =\, G_0\, \triangleleft \, G_1\, \triangleleft \, \cdots\,
\triangleleft \, G_m\, =\, G
\]
of $\alpha$-stable $C^k$-Lie subgroups
$G_j$ of~$G$ which cannot be properly refined
to a series of the same type
(let us call such a series
a \emph{Lie composition series}).
In other words, each factor
$G_j/G_{j-1}$ is a \emph{simple Lie
contraction group of class~$C^k$} in the sense that
it is non-trivial and does not
have a proper, normal, non-trivial
$C^k$-Lie subgroup
stable under the contractive $C^k$-automorphism
induced by~$\alpha$.\\[2.5mm]
We mention that Lie composition series
also exist if $(\K,|.|)$
is a non-locally compact, complete ultrametric field,
because $\dim_\K(G_j)<\dim_\K(G_{j+1})$
holds for the groups in a strictly
ascending Lie series.\\[2.5mm]
As a consequence of Theorem~B,
every simple Lie contraction group $(G,\alpha)$
of class $C^k$ (with $k\geq 2$)
over a complete ultrametric field
is abelian.
If $\car(\K)=0$ and $k=\omega$,
this easily implies that $G$
is isomorphic to $(\K^n,+)$ for some~$n$
and $\alpha$ corresponds to a $\K$-linear
automorphism.
If $\K$ has positive characteristic,
then currently we cannot say more.
%
\begin{problem}\label{probreuse}
Is it possible to classify
all simple Lie contraction groups
over complete ultrametric fields
of positive characteristic, or at least
over local fields
of positive characteristic\,?
\end{problem}
The following example shows
that simple Lie contraction groups
need not be $1$-dimensional --
each given dimension $n\in \N$ can occur.
\begin{example}\label{fromGeorge}
Let $\K:=\F_p(\!(X)\!)$,
$G:=(\K^2,+)$ and $\alpha\colon G\to G$
be defined via $(x,y)\mto (Xy,x)$.
Then the map $\phi\colon G\to \K=\F_p^{(-\N)}\times \F_p^{\N_0}$,
\[
\left({\textstyle \sum_k} a_kX^k,{\textstyle \sum_k} b_kX^k\right
)\, \mto \, {\textstyle \sum_k} a_k X^{2k}+{\textstyle \sum_k}
b_kX^{2k+1}
\]
is an isomorphism of topological groups
and $\phi\circ \alpha\circ \phi^{-1}$ is the right
shift on $\F_p^{(-\N)}\times \F_p^{\N_0}$.
Hence $(G,\alpha)$ is a simple
contraction group. As a Lie group,
$G$ is $2$-dimensional.\\[2.5mm]
Analogous arguments show
that the $n$-dimensional $\K$-analytic Lie group~$\K^n$,
together with $\alpha\colon \K^n\to\K^n$,
$\alpha(x_1,\ldots, x_n):=(X x_n,x_1,\ldots, x_{n-1})$,
is isomorphic
to $\F_p^{(-\N)}\times \F_p^{\N_0}$
and hence is a simple contraction group.
\end{example}
While Theorem~A\hspace*{.3mm} settles the locally compact case,
the following problem remains unsolved:
\begin{problem}
Is it true that all Lie contraction groups over
a non-locally compact, complete
ultrametric field $(\K,|.|)$
of positive characteristic are\linebreak
torsion groups\,?
\end{problem}
It would be enough to prove this for all
simple Lie contraction groups over~$\K$.\\[2.5mm]
Let us close this section with material
concerning Section~\ref{seclocglob}.
The following example shows that,
in the case of positive
characteristic,
analytic contraction groups
need not be determined
by the Lie algebra and its automorphism.
Not even the local structure of the group
is determined.
%
\begin{example}\label{ex221}
Let $\F$ be a finite field
and $\K:=\F(\!(X)\!)$.
We set $p:=\car(\F)$
and consider the $3$-dimensional
$\K$-analytic Lie group
$G:=\K^2\semid_\beta \K$,
where $\beta\colon \K\to \Aut(\K^2)$,
$z\mto \beta_z$ is the homomorphism
given by
\[
\beta_z(x,y)\; :=\; (x+z^py,y)\quad
\mbox{for $\, x,y,z\in \K$.}
\]
Then the map
\[
\alpha\colon G\to G\,,\quad
(x,y,z)\mto (X^{p+1}x,Xy,Xz)\quad
\mbox{for $x,y,z\in \K$}
\]
is a contractive
automorphism of~$G$.
Given $g=(x,y,z)$ and $h=(a,b,c)$ in~$G$,
we have $f(g,h):=ghg^{-1}h^{-1}=(z^pb-c^py,0,0)$.
Since $|z^pb|=O(\|g\|^2)O(\|h\|)$ as $(g,h)\to (0,0)$
and $|c^py|=O(\|h\|^2)O(\|g\|)$, it follows that
%
\begin{equation}\label{cfwith}
f(g,h)\; =\; o(\|(g,h)\|^2)
\end{equation}
(using Landau's big~$O$ and small~$o$-notation).
The
second
order Taylor expansion of
the commutator map $f$ of the $C^\omega$-Lie group~$G$
around $(0,0)$ reads
\[
f(g,h)\; =\; [g,h] + o(\|(g,h)\|^2)
\]
(cf.\ item~5 in \cite[Part~II, Chapter~IV, \S7]{Ser}).
Comparing with (\ref{cfwith}),
we deduce that $[g,h]=0$ for all $g,h\in \K^3$.
Thus $L(G)=\K^3$ is an abelian
Lie algebra.
Also, $L(\alpha)$ is the linear map
$\gamma\colon \K^3\to\K^3$, $(x,y,z)\mto (X^{p+1}x,Xy,Xz)$.
Now $(\K^3,+)$ is a $3$-dimensional
$\K$-analytic Lie group
admitting~$\gamma$ as a contractive
$\K$-analytic automorphism.
We have $L(\K^3)=L(G)$ and $L(\gamma)=L(\alpha)$.
However, $(\K^3,+)$ is abelian
while~$G$ is not
(and in fact $G$ does not even have
an abelian open subgroup).
\end{example}
One would not expect a positive solution to
the following existence question,
but the authors currently do not know
counterexamples.
\begin{problem}
If $\cg$ is a Lie algebra
over a local (or complete ultrametric) field~$\K$
of positive characteristic and
$\beta\colon \cg\to\cg$ a contractive Lie algebra
automorphism, can we always find
an analytic (or at least~$C^k$) Lie group~$G$
and an analytic (or $C^k$) contractive automorphism
$\alpha\colon G\to G$ such that
$\cg=L(G)$ and $\beta=L(\alpha)$\,?
\end{problem}
{\footnotesize
{\bf Helge Gl\"{o}ckner},
TU Darmstadt, Fachbereich Mathematik AG~5,
Schlossgartenstr.\,7,
64289 Darmstadt, Germany.
\,E-Mail: {\tt gloeckner\at{}mathematik.tu-darmstadt.de}}
\end{document}